\documentclass[11 pt]{amsart}

\usepackage[latin1]{inputenc}
\usepackage{amsmath,amsthm,amssymb,amscd,color, xcolor,mathtools,url}
\usepackage{bbm}

\newtheorem{thm}{Theorem}[section]
 \newtheorem{cor}[thm]{Corollary}
 \newtheorem{lem}[thm]{Lemma}
 \newtheorem{prop}[thm]{Proposition}

 \theoremstyle{definition}
 
 \theoremstyle{remark}
 \newtheorem{rem}[thm]{Remark}
 
 \numberwithin{equation}{section}

\def\be#1 {\begin{equation} \label{#1}}
\def\ee{\end{equation}}

\def\sqw{\hbox{\rlap{\leavevmode\raise.3ex\hbox{$\sqcap$}}$%
\sqcup$}}
\def\findem{\ifmmode\sqw\else{\ifhmode\unskip\fi\nobreak\hfil
\penalty50\hskip1em\null\nobreak\hfil\sqw
\parfillskip=0pt\finalhyphendemerits=0\endgraf}\fi}

\newcommand{\vecs}[2]{\vec{#2}^{(#1)}}

\newcommand{\R}{\mathbb R}

\newcommand{\Q}{\mathbb Q}
\newcommand{\N}{\mathbb N}
\newcommand{\Z}{\mathbb Z}

\title{Bounds for spectral projectors on tori}

\begin{document}

\begin{abstract}We investigate norms of spectral projectors on thin spherical shells for the Laplacian on tori. This is closely related to the boundedness of resolvents of the Laplacian, and to the boundedness of $L^p$ norms of eigenfunctions of the Laplacian. We formulate a conjecture, and partially prove it. \end{abstract}

\author[P. Germain and S. L. Rydin Myerson]{Pierre Germain and Simon L. Rydin Myerson}

\maketitle

\tableofcontents

\section{Introduction}

\subsection{Boundedness of spectral projectors on Riemannian manifolds}

\subsubsection{A general problem}
Given a Riemannian manifold $M$ with Laplace-Beltrami operator $\Delta$, and for $\lambda \geq 1$, $0<\delta <1$, let
$$
P_{\lambda,\delta} = P_{\lambda,\delta}^{\chi} = \chi \left( \frac{\sqrt{-\Delta} - \lambda}{\delta} \right).
$$
where $\chi$ is a non-negative cutoff function supported in $[-1,1]$, equal to $1$ on $[-\frac{1}{2},\frac{1}{2}]$.

A general question is to estimate
$$
{ \| P_{\lambda,\delta}^\chi \|_{L^2 \to L^p}, \qquad \mbox{where $p \in [2,\infty]$}}.
$$

Using self-adjointness of $P_{\lambda,\delta}^\chi$ and a $TT^*$ argument, it follows that
\begin{equation}
\label{mallard}
\| P^{\chi^2}_{\lambda,\delta} \|_{L^{p'} \to L^p} = \| P_{\lambda,\delta}^\chi \|_{L^{p'} \to L^2}^2 = \|  P_{\lambda,\delta}^\chi \|_{L^{2} \to L^p}^2.
\end{equation}
Furthermore, given two cutoff functions $\chi$ and $\widetilde{\chi}$, the boundedness of $P_{\lambda,\delta}^{\chi}$ on $L^2$ implies the following: if $\| P_{\lambda,\delta}^{\chi} \|_{L^2 \to L^p}$ obeys, say, a polynomial bound of the type $\lambda^\alpha \delta^\beta$, so does $\| P_{\lambda,\delta}^{\widetilde \chi} \|_{L^2 \to L^p}$, with a different constant. Therefore, it will be equivalent to estimate either of the three quantities appearing in~\eqref{mallard}, and the result is essentially independent of the cutoff function, which might even be taken to be a sharp cutoff.

Up to possibly logarithmic factors, this question is essentially equivalent to that of estimating the $L^2 \to L^p$ norm of the resolvent $R((x+iy)^2) = (\Delta + (x+iy)^2)^{-1}$; this is the point of view taken in  Dos Santos Ferreira-Kenig-Salo~\cite{DKS} and Bourgain-Shao-Sogge-Yao \cite{BSSY}. Essentially, one can think of $R((x+iy)^2)$ as a variant of $\frac{1}{xy} P_{x,y}$.

\subsubsection{The case of Euclidean space} \label{sec:euclidean-space}

We will denote the Stein-Tomas exponent
$$
p_{ST} = \frac{2(d+1)}{d-1}.
$$
As will become clear, it often plays the role of a critical point when estimating the norm of $P_{\lambda,\delta}$.

On $\mathbb{R}^d$ (with the Euclidean metric), there holds
\begin{equation}
\label{swallow}
\| P_{\lambda,\delta} \|_{L^2 \to L^p} \lesssim
\left\{
\begin{array}{ll}
\lambda^{\sigma(p)/2} \delta^{1/2} & \mbox{if $p \geq p_{ST}$} \\
\lambda^{\frac{d-1}{2} \left( \frac{1}{2} - \frac{1}{p} \right)} \delta^{\frac{(d+1)}{2}\left( \frac{1}{2} - \frac{1}{p} \right)} & \mbox{if $2 \leq p \leq p_{ST}$},
\end{array}
\right.
\end{equation}
where 
$$
\sigma(p) = d - 1 - \frac{2d}{p} \qquad \mbox{so that} \qquad \sigma(p_{ST}) = \frac{d-1}{d+1};
$$
see the appendix for a proof of the above bounds. For more general second order operators in the resolvent formulation, we refer to Kenig-Ruiz-Sogge~\cite{KRS}. Finally, the case of the hyperbolic space was recently treated by the first author and L\'eger~\cite{GL}.

\subsubsection{The case of a compact manifold} On a compact manifold of dimension $d$, as was proved by Sogge~\cite{Sogge},
$$
\| P_{\lambda,1} \|_{L^2 \to L^p} \lesssim
\left\{
\begin{array}{ll}
\lambda^{\sigma(p)/2} & \mbox{if $p \geq p_{ST}$} \\
\lambda^{\frac{d-1}{2} \left( \frac{1}{2} - \frac{1}{p} \right)} & \mbox{if $2 \leq p \leq p_{ST}$},
\end{array}
\right.
$$
where $\sigma(p)$ is as above.
For any given compact manifold, this estimate is optimal for $\delta = 1$. In the case of the sphere $ \mathbb{S}^d$ (or more generally of a Zoll manifold), it does not improve if $\delta$ decreases, since the eigenvalues of the sphere Laplacian are essentially distributed like squared integers. However, for ``most" manifolds, the estimates above are expected to improve as $\delta$ decreases. It is the aim of this article to examine this question in the case of the torus.

If the manifold $M$ is negatively curved, then logarithmic improvements are possible over the allowed range of $\delta$, as in Bourgain-Shao-Sogge-Yao \cite{BSSY} and Blair-Sogge~\cite{BS}. The work of Sogge-Toth-Zelditch~\cite{STZ} shows that generic manifolds also allow improvements.

\subsection{Spectral projectors on tori}

\subsubsection{Formulating the problem}

From now on, we focus on the case of tori given by the quotient $\mathbb{R}^d / (\mathbb{Z} e_1 + \dots + \mathbb{Z} e_1)$, where $e_1,\dots,e_d$ is a basis of $\mathbb{R}^d$, with the standard metric. This is equivalent to considering the operators
$$
P_{\lambda,\delta} = \chi \left( \frac{\sqrt{-Q(\nabla)} - \lambda}{\delta} \right) \qquad \mbox{on} \;\; \mathbb{T}^d = \mathbb{R}^d / \mathbb{Z}^d,
$$
where $\nabla$ is the standard gradient operator, and $Q$ is a positive definite quadratic form on $\mathbb{R}^d$, with coefficients $\beta_{ij}$: 
$$
Q(x) = \sum_{i=1}^d \beta_{ij} x^i x^j \qquad \implies \qquad Q(\nabla) = - \sum_{i=1}^d \beta_{ij} \partial_i \partial_j.
$$
Dispensing with factors of $2\pi$, which can be absorbed in $Q$, the associated Fourier multiplier has the symbol
$$
\chi \left( \frac{\sqrt{Q(k)} - \lambda}{\delta} \right).
$$

\subsubsection{Known results for $p = \infty$: counting lattice points}
\label{subsec1}
Abusing notations by writing $P_{\lambda,\delta}(z)$ for the convolution kernel giving $P_{\lambda,\delta}$, we have the formula
$$
P_{\lambda,\delta}(z) = \sum_n \chi \left( \frac{\sqrt{Q(n)} - \lambda}{\delta} \right) e^{2\pi i n \cdot z}.
$$
It is easy to see that 
$$
\| P_{\lambda,\delta} \|_{L^1 \to L^\infty} = \| P_{\lambda,\delta}(z) \|_{L^\infty_{z}} = \sum_n \chi \left( \frac{\sqrt{Q(n)} - \lambda}{\delta} \right).
$$
If $\chi = \mathbf{1}_{[-1,1]}$, this can be expressed as
$$
\| P_{\lambda,\delta} \|_{L^1 \to L^\infty}  = N(\lambda + \delta) - N(\lambda - \delta),
$$
where $N(\lambda)$ is the counting function associated to the quadratic form $Q$: namely, it denotes the number of lattice points $n \in \mathbb{Z}^d$ such that  $Q(n) < \lambda^2$. To leading order, $N(\lambda)$ equals $\operatorname{Vol}(B_1) \lambda^d$, where $\operatorname{Vol}(B_1)$ is the volume of the ellipsoid $\{Q(x) < 1\}.$ We denote the error term by $P(\lambda)$, thus:
$$
N(\lambda) = \operatorname{Vol}(B_1) \lambda^d + P(\lambda).
$$

For a general quadratic form $Q$, it was showed by Landau~\cite{Landau15} that $P(\lambda) = O( \lambda^{d - \frac{2d}{d+1}})$. \textcolor{black} {Consequently, we have
\begin{align}\label{eqn:landau_shells}
\| P_{\lambda,\delta} \|_{L^1 \to L^\infty} &\ll \delta\lambda^{d-1}
&\text{for }\delta&> \lambda^{-\frac{d-1}{d+1}}. 
\end{align}}
Landau's result, and hence the range for \(\delta\) in \eqref{eqn:landau_shells}, has been improved for every dimenson \(d\). Nonetheless \eqref{eqn:landau_shells} is a useful point of comparison since our approach is in a sense a refinement of a proof of Landau's theorem (see the comments after Theorem~\ref{thmcaps}).
Regarding lower bounds for \(P\), when \(Q(x)=|x|_2^2\) one can show that \(P(\lambda_i)\gg \lambda_i^{d-2}\) for some sequence \(\lambda_i \to \infty\). The present state of the art is as follows:
\begin{itemize}
	\item If \(d=2\) then estimating $P(\lambda)$ is a variation on the celebrated Gauss circle problem. One conjectures \(P(\lambda)=O_\epsilon(\lambda^{\frac{1}{2}+\epsilon})\), and the best known result is  \(O(\lambda^{\frac{131}{208}}\log^{\frac{18627}{8320}}\lambda)\), see Huxley~\cite{Huxley03}.
	\item If \(d=3\) then one conjectures \(P(\lambda)=O_\epsilon(\lambda^{1+\epsilon})\), see  Nowak~\cite[\S\S1.1-1.2]{Nowak2014}. We have \(O(\lambda^{\frac{231}{158}})\) by Guo \cite{Guo2012}. If moreover \(Q\) has rational coefficients then \(P(\lambda)=O(\lambda^{\frac{21}{16}})\) by Chamizo-Cristob\'al-Ubis~\cite{CCU}. 
	\item If \(d=4\) then \(P(\lambda)= O(\lambda^2 \log^{\frac{2}{3}}\lambda)\) by Walfisz~\cite{walfisz1960}. The case \(Q(x)=|x|_2^2\) shows that up the log power this is best possible.
	\item If \(d>4\) then $P(\lambda) = O( \lambda^{d - 2})$, see Kr\"atzel \cite{Kraetzel00}. This is best possible  if \(Q\) is  a multiple of a form with rational coefficients, and if not then $P(\lambda) = o( \lambda^{d - 2})$ by G\"otze~\cite{Goetze}.
\end{itemize}

\subsubsection{Known results on standard tori: eigenfunctions of the Laplacian}
\label{subsec2} It was conjectured by Bourgain~\cite{Bourgain} that an eigenfunction $f$ of the Laplacian on the standard torus with eigenvalue $\lambda^2$ satisfies
$$
\| f \|_{L^p} \lesssim_\epsilon \lambda^{\frac{d-2}{2} - \frac{d}{p} + \epsilon} \| f \|_{L^2} \qquad \mbox{for $p \geq p^*$}, \;\;\;\mbox{where} \; p^* = \frac{2d}{d-2},
$$
which can be reformulated as
$$
\| P_{\lambda,\frac{1}{\lambda}} \|_{L^2 \to L^p} \lesssim_\epsilon \lambda^{\frac{d-2}{2} - \frac{d}{p} + \epsilon} \qquad \mbox{for $p \geq p^*$.}
$$
Progress towards this conjecture~\cite{Bourgain2, BourgainDemeter1, BourgainDemeter2} culminated in the work of Bourgain and Demeter on $\ell^2$-decoupling~\cite{BourgainDemeter3}, where the above conjecture is proved for $d \geq 4$ and $p \geq \frac{2(d-1)}{d-3}$.

\subsubsection{Known results on standard tori: uniform resolvent bounds} 
\label{subsec3} It was proved in Dos Santos Ferreira-Kenig-Salo~\cite{DKS}  that, for general compact manifolds, each $x,y \in \mathbb{R}$, and writing $ p^* = \frac{2d}{d-2}$, we have
$$
\| (\Delta + (x+iy)^2)^{-1} \|_{L^{(p^*)'} \to L^{p^*}} \lesssim 1 \qquad \mbox{if $|y| \geq 1$}.
$$
In terms of spectral projectors, this is equivalent (see Cuenin~\cite{Cuenin}) to the bound
$$
\| P_{\lambda,\delta} \|_{L^{(p^*)'} \to L^{p^*}}  \lesssim \lambda \delta, \qquad \mbox{if $|\delta| > 1$}.
$$

It was also asked whether this bound could be extended to a broader range of $y$, or equivalently a broader range of \(\delta\). Bourgain-Shao-Sogge-Yao \cite{BSSY} showed that on the sphere the range above is optimal. In the case of the standard \(d\)-dimensional torus \(\R^d/\Z^d\), they could improve earlier results of Shen~\cite{Shen}. The results of Shen and  Bourgain-Shao-Sogge-Yao were then sharpened by Hickman~\cite{Hickman}, who extended the range for the standard \(d\)-dimensional torus futher to  $|\delta| > \lambda^{-\frac{1}{3} -\frac{d}{3(21d^2-d-24)}  + \epsilon}$. 

\subsubsection{Known results in dimension 2} The classical estimate of Zygmund corresponds, in our language, to a sharp result for $d=2$, $p=4$, $\delta = \lambda^{-1}$. It was showed by Bourgain-Burq-Zworski~\cite{BBZ} that it can be extended to $\delta > \lambda^{-1}$. 
A striking feature of the estimates in~\cite{BBZ} is that they entail no subpolynomial loss ($\epsilon$ power in the exponent), which has important consequences for control theory in particular, as explained in that paper.

\subsection{Conjecture and results}

Based on two specific examples, developed in Section~\ref{lbac}, we conjecture that
\begin{equation}
\label{conj}
\boxed{
\| P_{\lambda,\delta} \|_{L^{2} \to L^p} \lesssim (\lambda \delta)^{\frac{(d-1)}{2} \left( \frac{1}{2} - \frac{1}{p} \right)} + \lambda^{\sigma(p)/2} \delta^{1/2},}
\end{equation}
where $\delta > \lambda^{-1}$, and $\sigma(p) = d - 1 - \frac{2d}{p}$, for any fixed torus. We show there that this bound would be optimal, and describe when each term in the conjecture dominates.

\bigskip

The methods developed in the present paper give improvements on the range of validity of this conjecture. The precise statemenet is Theorem~\ref{thm:main} below. \textcolor{black}{As this is a rather cumbersome formula, we choose to state some simpler results in fairly natural cases of interest, namely $p<p_{ST}$, \(\delta\) large,  and $d=3$. Here we have $p_{ST} = \frac{2(d+1)}{d-1}$; in sections~\ref{subsec2} and~\ref{subsec3} we saw that  $p^* = \frac{2d}{d-2}$ has some special significance, so we will also state a result in this case.}

\begin{thm}[The case $p<p_{ST}$] \label{thmpST} 
For any positive definite quadratic form $Q$, the conjecture~\eqref{conj} is verified, up to subpolynomial losses, if $\lambda>1$, \(\delta \geq \lambda^{-1}\) and $1<p<p_{ST}$. 
\end{thm}

Here, subpolynomial losses means that the conjecture holds true with an additional $\lambda^\epsilon$ factor on the right-hand side, where the implicit constant depends on $\epsilon$, but $\epsilon$ can be chosen arbitrarily small.

\textcolor{black}{
\begin{thm}[The case of large $\delta$] \label{thmsimple} 
	For any positive definite quadratic form $Q$, the conjecture~\eqref{conj} is verified, up to subpolynomial losses, if $\lambda>1$, $p\geq p_{ST}$ and
	\[
	\delta > 
	\lambda^{-\frac{(d-1)p-dp_{ST}+2}{(d+1)p-dp_{ST}-2}}
	.\]
\end{thm}
}
By substituting \(p=2d/(d-2)\) and performing a brief computation we obtain:
\begin{cor}[The case $p = p^*$] \label{thmp*} \textcolor{black}{Let \(p^*=\frac{2d}{d-2}\).} For any positive definite quadratic form $Q$, the conjecture~\eqref{conj} is verified, up to subpolynomial losses, for $p=p^*$, if $\lambda>1$ and \textcolor{black}{\(\delta \geq \lambda^{-\frac{1}{2 d - 1}}\)}.\end{cor}

\textcolor{black}{In Theorems~\ref{thmsimple} and Corollary~\ref{thmp*} we have aimed to provide simple statements, which are consequently somewhat weaker than Theorem~\ref{thm:main} below. For any particular \(d\) these last results can be improved by a short computation. We present the following as a representative example.}

\begin{thm}[The case $d=3$] \label{thmd3} For any positive definite quadratic form $Q$, the conjecture~\eqref{conj} is verified, up to subpolynomial losses, if $d=3$, whenever $\lambda>1$, \textcolor{black}{\(\delta \geq \min\{\lambda^{-\frac{3p-8}{5p-8}}, \lambda^{-\frac{8-p}{5p-16}}\}\) and also \(\delta \geq\lambda^{-1/2}\)}.
	\end{thm}

\textcolor{black}{In the proofs of the results above, the value  \(\delta = \lambda^{-\frac{d-1}{d+1}}\) will emerge as playing a special role. In particular, to prove our conjecture in even a single case with \(p>p_{ST}\) and  \(\delta \ll \lambda^{-\frac{d-1}{d+1}}\)   needs a  different approach. This threshold also appears in the classical result \eqref{eqn:landau_shells}, and more generally when counting lattice points in a \(\delta\)-thick shell around  a manifold with curvature \(\sim \lambda^{-1}\) using for example Poisson summation. Substituting \(d=3\) into the last theorem does however yield the full range \(\delta > \lambda^{-\frac{d-1}{d+1}}\), as well as the the full range  \(p>2\), in the following setting:}

\begin{cor}If \(d=3\), then for any positive definite quadratic form $Q$ the conjecture holds for all   \(\delta \geq \lambda^{-\frac{d-1}{d+1}} = \lambda^{-1/2}\) if \(p\leq p_{ST}+\frac{4}{7}\) \textcolor{black}{or \(p\geq p_{ST}+4\)},  and it holds for all \(p>2\) if  $\delta> \lambda^{-2/5}$.
\end{cor}

The proof of the above results will combine a number theoretical argument, which allows one to count the number of caps in a spherical shell which contain many lattice points, with a harmonic analysis approach, relying in particular on the $\ell^2$ decoupling theorem of Bourgain and Demeter.

\bigskip

In order to understand better the statement of these theorems, it is helpful to spell out what they imply for each of the classical problems presented in sections~\ref{subsec1},~\ref{subsec2},~\ref{subsec3}.
\begin{itemize}
\item \textcolor{black}{For the problem of counting points in thin spherical shells (Subsection~\ref{subsec1}) we recover the bound \eqref{eqn:landau_shells} of Landau~\cite{Landau15}, see also the comments after Theorem~\ref{thmcaps}.}
\item \textcolor{black}{The problem of bounding $L^p$ norms of eigenfunctions was prevously considered for rational tori, that is \(\R^d/A\Z^d\) where \(A\in \operatorname{GL}_d(\Q)\). Our results do not improve the bounds of Bourgain-Demeter~\cite{BourgainDemeter3} in this case. For generic tori, eigenfunction bounds are trivial;  the natural analogue of bounding the $L^p$ norms of eigenfunctions is to bound the operator norm of $P_{\lambda,\frac{1}{\lambda}}$, and this question does not appear to have been considered before. For any torus, that is any \(\R^d/B\Z^d\) with \(B\in \operatorname{GL}_d(\R)\), we obtain from Theorem~\ref{thm:main} below the bound
\begin{align*}
	\| P_{\lambda,\frac{1}{\lambda}} \|_{L^2 \to L^p} &\lesssim_\epsilon \lambda^\epsilon
	(\lambda^{\frac{d}{d+1}})^{\frac{1}{2}(1-\frac{2}{p})+\frac{d}{2}(1- \frac{p_{ST}}{p})-\sqrt{(1 - \frac{2}{p})
			(1- \frac{p_{ST}}{p})}}
	&
	(p&\geq p_{ST}).
\end{align*}
}

\item For the problem of proving uniform resolvent bounds (Subsection~\ref{subsec3}), it proves the desired estimate up to a subpolynomial loss
$$\| P_{\lambda,\delta} \|_{L^{(p^*)'} \to L^{p^*}}  \lesssim_\epsilon \lambda^{1+\epsilon} \delta, \qquad $$
if \( \delta \geq \lambda^{-1+\frac{2 d}{5 d - 4}}\), 
or if \(d=3\) and \(\delta > \lambda^{-1/2}\),  improving over \textcolor{black}{Hickman's~\cite{Hickman}  result that} $\lambda^{-\frac{1}{3} -\frac{d}{3(21d^2-d-24)}}$.

\end{itemize}
\subsection{Acknowledgments}
The authors are grateful to the anonymous referee for a careful reading of their manuscript, and pointing out an error in an earlier version; they are also thankful to Yu Deng for insightful discussions at an early stage of this project.

While working on this project SLRM was  supported by DFG project number 255083470, and by a Leverhulme Early Career Fellowship. PG was supported by the NSF grant DMS-1501019, by the Simons collaborative grant on weak turbulence, and by the Center for Stability, Instability and Turbulence (NYUAD).

\section{Notation}\label{sec:notation}  Throughout, \textcolor{black}{$p_{ST} = \frac{2(d+1)}{d-1},$ and $\sigma(p) = d - 1 - \frac{2d}{p}$ will be as in section~\ref{sec:euclidean-space}, and  $p^* = \frac{2d}{d-2}$ as in section~\ref{subsec2}.}
We adopt the following normalizations for the Fourier series on $\mathbb{T}^d$ and Fourier transform on $\mathbb{R}^d$, respectively:
\begin{align*}
& f(x) = \sum_{k \in \mathbb{Z}^d} \widehat{f}_k e^{2\pi i k \cdot x}, \qquad \qquad \widehat{f}_k = \int_{\mathbb{T}^d} f(x) e^{-2\pi i k \cdot x} \,dx \\
& f(x) = \int_{\mathbb{R}^d} \widehat{f}(\xi) e^{2\pi ix \cdot \xi} \,dx, \qquad \qquad \widehat{f}(\xi) = \int_{\mathbb{R}^d} f(x) e^{-2\pi ix \cdot \xi} \,dx 
\end{align*}
The Poisson summation formula is then given by
$$
\sum_{n \in \mathbb{Z}^d} f(n) = \sum_{k \in \mathbb{Z}^d} \widehat{f}(k).
$$

We write \((\vecs{1}{v}|\cdots|\vecs{k}{v})\) for the matrix with columns \(\vecs{i}{v}\).

Given two quantities $A$ and $B$, we write $A \lesssim B$ or equivalently \(A = O(B)\) if there exists a constant $C$ such that $A \leq CB$, and $A \lesssim_{a,b,c} B$ if the constant $C$ is allowed to depend on $a,b,c$. We always allow \(C\) to depend on the dimension \(d\). In the following, it will often be the case that the implicit constant will depend on $\beta$, and on an arbitrarily small power of $\lambda$: $A \lesssim_{\beta,\epsilon} \lambda^\epsilon B$. When this is clear from the context, we simply write $A \lesssim \lambda^\epsilon B$. When we are assuming that the implicit constant is sufficiently small, we will write \(A\ll B\).

If both \(A\lesssim B\) and \(B\lesssim A\) then we write \(A\sim B\).

\section{Lower bounds and conjecture}

\label{lbac}

\subsection{The discrete Knapp example}

\begin{lem} \label{llb1} For any $n \in \mathbb{N}$, there exists $\lambda \sim |n|$ such that
if $\delta \in (0,1)$, 
$$
\| P_{\lambda,\delta} \|_{L^{2} \to L^p} \gtrsim (1 + \lambda \delta)^{\frac{(d-1)}{2} \left( \frac{1}{2} - \frac{1}{p} \right)}.
$$
\end{lem}

\begin{proof} Consider the ellipse $\{ \xi \in \mathbb{R}^d, \; Q(\xi) = 1\}$. Its normal vector is colinear to $e_d = (0,\dots,0,1)$ at the point $\xi_0$. We now dilate this ellipse by a factor $\lambda$ such that $\lambda \xi^d_0 = n \in \mathbb{N}$, smear it to a thickness $\delta$, and observe that, around the point $\lambda \xi_0$, it contains many lattice points of $\mathbb{Z}^d$. More precisely, the cuboid $C \subset \mathbb{R}^d$ defined by
$$
C = \{ |\xi^i - \lambda \xi^i_0| < c \sqrt{\lambda \delta} \;\; \mbox{for $i = 1,\dots,d-1$} \quad \mbox{and} \quad  |\xi^d - \lambda \xi^d_0| < c\delta \}
$$
(where the constant $c$ is chosen to be sufficiently small) is such that
$$
C \subset \left\{ \xi \in \mathbb{R}^d, \,\lambda-\frac{\delta}{2} < \sqrt{Q(\xi)} < \lambda + \frac{\delta}{2} \right\}.
$$
Furthermore, for $\lambda\in\mathbb{Z}$, $C$ will contain $\sim (1 + \lambda \delta)^{\frac{d-1}{2}}$ points in $\mathbb{Z}^d$. Writing $\xi = (\xi',\xi^d)$ and $x=(x',x^d)$, let
$$
f(x) = e^{2\pi i\lambda \xi_0^d x^d}\sum_{\xi' \in \mathbb{Z}^{d-1}} \phi \left( \frac{\xi'- \lambda \xi_0'}{\sqrt{\lambda \delta}} \right) e^{2\pi i \xi' \cdot x'}, 
$$
where $\phi \in \mathcal{C}_0^\infty (\mathbb{R}^{d-1})$ is such that $\widehat{\phi} \geq 0$ and $\operatorname{Supp} \widehat{\phi} \subset C$. By the Poisson summation formula, $f$ can be written
$$
f(x) = e^{2\pi i\lambda x^d  \xi_0^d} (\lambda \delta)^{\frac{d-1}{2}} \sum_{n \in \mathbb{Z}^{d-1}} \widehat{\phi}( \sqrt{\lambda \delta}(x'-n)) e^{2 \pi i \lambda \xi_0' \cdot (x'-n)},
$$
which implies
$$
\| f \|_{L^p} \sim (1 + \lambda \delta)^{\frac{d-1}{2} \left( 1 - \frac{1}{p} \right)}.
$$
Since $P_{\lambda,\delta} f = f$, we find that
$$
\| P_{\lambda,\delta} \|_{L^{2} \to L^p} \geq \frac{\| f \|_{L^p}}{\| f \|_{L^{2}}} \sim (1 + \lambda \delta)^{\frac{(d-1)}{2} \left( \frac{1}{2} - \frac{1}{p} \right)}.
$$
\end{proof}

\subsection{The radial example} 
\begin{lem} 
\label{llb2}
For any $n \in \mathbb{N}$ and $\delta \in (0,1)$, there exists $\lambda$ such that $|n-\lambda| \lesssim 1$ and
$$
\| P_{\lambda,\delta} \|_{L^2 \to L^p} \gtrsim \lambda^{\sigma(p)/2} \sqrt \delta.
$$
\end{lem}

\begin{proof}
For any $n \in \mathbb{N}$, there exists $\lambda$ with $|n-\lambda| \lesssim 1$, and such that the corona
$$
\mathcal{C} = \{ \lambda -\frac{\delta}{2} < Q(x) < \lambda + \frac{\delta}{2} \}
$$
contains $N \gtrsim \lambda^{d-1} \delta$ points in $\mathbb{Z}^d$. Define
$$
f(x) = \sum_{\xi \in \mathcal{C} \cap \mathbb{Z}^{d-1}} e^{2\pi i \xi \cdot x}.
$$
It is clear that
$$
\| f \|_{L^\infty} = N \qquad \mbox{while} \qquad \| f \|_{L^2} = \sqrt{N}.
$$
By Bernstein's inequality, for $p\geq 2$,
$$
\| f \|_{L^p} \gtrsim \|f\|_{L^{\infty}} \lambda^{-\frac{d}{p}} \sim \lambda^{-\frac{d}{p}} N.
$$
Therefore, 
$$
\| P_{\lambda,\delta} \|_{L^2 \to L^p} \geq \frac{\|f\|_{L^p}}{\|f\|_{L^2}} \gtrsim \lambda^{\sigma(p)/2} \sqrt \delta.
$$
\end{proof}

\subsection{The conjecture} Based on lemmas~\ref{llb1} and~\ref{llb2}, it is reasonable to conjecture that
$$
\boxed{\| P_{\lambda,\delta} \|_{L^{2} \to L^p} \lesssim (\lambda \delta)^{\frac{(d-1)}{2} \left( \frac{1}{2} - \frac{1}{p} \right)} + \lambda^{\sigma(p)/2} \delta^{1/2}.}
$$
The next question is: how small can $\delta$ be taken? In full generality, the limitation is 
$$
\delta \geq \frac{1}{\lambda},
$$
as this is best possible for rational tori; this will be the range we consider here.

We now describe the different regimes involved in the above conjecture.

\bigskip \noindent \underline{If $d=2$}, the conjecture can be formulated as
\begin{itemize}
\item $\| P_{\lambda,\delta} \|_{L^{2} \to L^p} \lesssim (\lambda \delta)^{\frac{1}{4} - \frac{1}{2p} }$ if $\left\{ \begin{array}{l} 2 \leq p \leq 6\; \mbox{and} \; \delta > \lambda^{-1} \\ \mbox{or} \; p \geq 6 \; \mbox{and} \; \lambda^{-1} < \delta < \lambda^{\frac{6-p}{2+p}} \end{array} \right.$.
\item $\| P_{\lambda,\delta} \|_{L^{2} \to L^p} \lesssim \lambda^{\frac 1 2 - \frac{2}{p}} \delta$ if $p \geq 6$ and $\delta > \lambda^{\frac{6-p}{2+p}}$.
\end{itemize}

\bigskip \noindent \underline{If $d\geq 3$}, let
$$
p_{ST} = \frac{2(d+1)}{d-1}, \qquad p^* = \frac{2d}{d-2}, \qquad  \widetilde p = \frac{2(d-1)}{d-3}
$$
and define
$$
e(p) = \frac{d+1}{d-1} \cdot \frac{\frac{1}{p} - \frac{1}{p_{ST}}}{\frac{1}{p} - \frac{1}{\widetilde{p}}}
$$
We have
$$
p_{ST} < p^* < \widetilde{p}.
$$

Keeping in mind that $\delta > \lambda^{-1}$, the above conjecture becomes
\begin{itemize}
\item $\| P_{\lambda,\delta} \|_{L^{2} \to L^p} \lesssim (\lambda \delta)^{\frac{(d-1)}{2} \left( \frac{1}{2} - \frac{1}{p} \right)}$ if 
$\left\{ \begin{array}{l} 2 \leq p \leq p_{ST} \\ \mbox{or} \; p_{ST} \leq p \leq p^* \; \mbox{and} \; \delta < \lambda^{e(p)} \end{array} \right.$
\item $\| P_{\lambda,\delta} \|_{L^{2} \to L^p} \lesssim \lambda^{\sigma(p)/2} \delta^{1/2}$ if $\displaystyle \left\{ \begin{array}{l} p_{ST} \leq p \leq p^* \; \mbox{and} \; \delta > \lambda^{e(p)} \\ \mbox{or} \; p \geq p^* \end{array} \right.$
\end{itemize}

\section{Caps containing many points}

We split the spherical shell
$$
S_{\lambda,\delta} = \{ x \in \mathbb{R}^d, \, \left| \sqrt{Q(x)}-\lambda \right| < \delta \}
$$
into a collection $\mathcal{C}$ of almost disjoint caps $\theta$:
$$
S_{\lambda,\delta} = \bigcup_{\theta \in \mathcal{C}} \theta,
$$
where each cap is of the form
$$
\theta = \{ x \in \mathbb{R}^d, \, | x - x_\theta| < \sqrt{\lambda \delta} \} \cap S_{\lambda,\delta} \qquad \mbox{for some $x_\theta \in S_{\lambda,\delta}$}.
$$
Each cap fits into a rectangular box with dimensions $\sim \delta \times \sqrt{\lambda \delta} \times \dots \times \sqrt{\lambda \delta}$. We call \(\vec{n}_\theta = \frac{x_\theta}{|\vec{x}_\theta|_2}\) the \emph{normal vector} to \(\theta\); observe that as \(\theta\) varies over caps, the normal vector \(\vec{n}_\theta\) varies over a \(\sqrt{\delta/\lambda}\)-spaced set.

Denote $N_\theta$ for the number of points in $\mathbb{Z}^d \cap \theta$. On the one hand, it is clear that $N_\theta \lesssim (\sqrt{\lambda \delta})^{d-1}$. On the other hand, one expects that the average cap will contain a number of points comparable to its volume, in other words $N_\theta \sim (\sqrt{\lambda \delta})^{d-1} \delta$ (provided this quantity is $>1$, which occurs if $\delta > \lambda^{-\frac{d-1}{d+1}}$).

This leads naturally to defining the following sets, which gather caps containing comparable numbers of points
\begin{align*}
& \mathcal{C}_0 = \{ \theta \in \mathcal{C}, \, N_\theta < (\sqrt{\lambda \delta})^{d-1} \delta \} \\
& \mathcal{C}_j =  \{ \theta \in \mathcal{C}, \,  (\sqrt{\lambda \delta})^{d-1} \delta 2^{j-1}<N_\theta\leq  (\sqrt{\lambda \delta})^{d-1} \delta 2^{j}\} , \qquad \mbox{for $1\leq 2^j \lesssim \delta^{-1}$ }.
\end{align*}

\begin{thm}
\label{thmcaps}
There is a constant \(K>0\), depending only on \(d\), as follows.
Let \( k \in \{ 1, \dotsc, d-1 \} \). If \( 2^{ j } >K \) and $(\sqrt{ \delta \lambda })^{ k } \delta 2^{ j } >K$ then
\[
\# \mathcal{ C }_{ j }
\lesssim
( 2^{ j / k } \delta )^{ - d }.
\]
\end{thm}

\begin{rem}
\textcolor{black}{There is an integer $k \in \{ 1, \dotsc, d-1 \}$ satisfying $(\sqrt{ \delta \lambda })^{ k } \delta 2^{ j } >K$ whenever $\delta > \lambda^{-\frac{d-1}{d+1}}$, which in practise we will assume whenever we apply the  theorem above.}
\end{rem}

\textcolor{black}{
By summing over the caps in each \(  \mathcal{C}_j \), we find that \( \#(S_{\lambda,\delta}\cap \Z^d) \ll \delta \lambda^{d-1} +\sqrt{\lambda/\delta}^{d-1} \), and in particular \( \#(S_{\lambda,\delta}\cap \Z^d) \ll \delta \lambda^{d-1}\) for \(\delta > \lambda^{-\frac{d-1}{d+1}}\). This recovers the  classical result \eqref{eqn:landau_shells}. As explained in section~\ref{subsec1}, the range  \(\delta > \lambda^{-\frac{d-1}{d+1}}\) has now been improved in every dimension. Thus Theorem~\ref{thmcaps} is certainly suboptimal if  \(\delta \ll \lambda^{-\frac{d-1}{d+1}}\).}

To further gauge the strength of this theorem we can compare it to the trivial bounds 
\begin{align}
\# \mathcal{C}_0 &\lesssim ( \lambda \delta^{-1})^{\frac{d-1}{2}},\label{eqn:C-zero}
\\
\# \mathcal{C}_j &=0\qquad(
2^j \gg \delta^{-1}).
\label{eqn:C-max}
\end{align}
	If $\delta = \lambda^{-\frac{d-1}{d+1}}$ then the  theorem interpolates between these bounds. For  it  reduces,  on the one hand, to \(\#\mathcal{C}_j \lesssim ( \lambda \delta^{-1})^{\frac{d-1}{2}}\) for \(2^j\sim 1\), and on the other, to $\# \mathcal{C}_j \lesssim (\delta 2^j)^{-d}$ for $\delta 2^j \gg (\sqrt{\lambda\delta})^{-1}$.

	If \(\delta > \lambda^{-\frac{d-1}{d+1}}\) then the theorem shows that for some constant \(C_d>0\), all but \(\frac{C_d}{\delta (\sqrt{\lambda\delta})^{d-1}}\%\) of the caps \(\theta\) satisfy \(N_\theta \lesssim\delta (\sqrt{\lambda\delta})^{d-1}\), and interpolates between this bound and \eqref{eqn:C-max}. Inspecting the proof, we could strengthen the former bound: if  \(\delta > \lambda^{-\frac{d-1}{d+1}}\) then all but \(\frac{C_d}{\delta (\sqrt{\lambda\delta})^{d-1}}\%\) of the caps \(\theta\) contain a fundamental region for \(\Z^d\). 
	
	If \(\delta < \lambda^{-\frac{d-1}{d+1}}\)  then the theorem shows that all but \((C'_d\sqrt{\lambda}(\sqrt{\delta})^{\frac{d+1}{d-1}})\%\) of the caps \(\theta\) satisfy \(N_\theta \lesssim 1\), and interpolates between this bound and \eqref{eqn:C-zero}. Again by inspecting the proof we could strengthen first part: if  \(\delta > \lambda^{-\frac{d-1}{d+1}}\) then all but \((C'_d\sqrt{\lambda}(\sqrt{\delta})^{\frac{d+1}{d-1}})\%\) of the caps \(\theta\) satisfy \(N_\theta \leq 1\). In this regime most caps should have \(N_\theta=0\).

\begin{proof}[Proof of Theorem~\ref{thmcaps}]
Let \(\theta\) be any cap. Let \(R_\theta\) be a rectangular box, centred at the origin, containing \(\theta-\theta\) and having dimensions $\sim \delta \times \sqrt{\lambda \delta} \times \dots \times \sqrt{\lambda \delta}$.

Define a norm \(|\,\cdot\,|_\theta\) on \(\R^d\) by
\[
|\vec{x}|_\theta
= \inf\{r>0 : \vec{x}\in r R_\theta\},
\]
so that \(R_\theta\) is the unit ball in this norm, with
\begin{equation}\label{eqn:theta-norm}
|\,\cdot\,|_2\lesssim \sqrt{\lambda\delta}|\,\cdot\,|_\theta.
\end{equation}
Because \(R_\theta\) is contained in a slab of the form \(\{\vec{x}\in\R^d:\vec{n}_\theta\cdot \vec{x}\lesssim \delta\}\), we also  have
\begin{equation}\label{eqn:theta-norm-normal}
	\vec{n}_\theta \cdot \vec{x}
	\lesssim
	\delta |\vec{x}|_\theta.
\end{equation}

Morally, the idea of the proof is to fix the lattice generated by \(R_\theta\cap \Z^d\), and count the number of caps with a given lattice. Carrying out this programme in a literal fashion seems possible but technically complex. We take advantage of a trick: we will construct  a small integer vector \(\vec{v}\) which is orthogonal to all the integer vectors in \(R_\theta\cap \Z^d\) and approximately perpendicular to \(\vec{n}_\theta\), and it is this vector, rather than the  lattice itself, which we will fix. The outline of the proof is as follows: after some further definitions we set out some basic results from the geometry of numbers in Step~1 below; we then construct \(\vec{v}\) in Step~2 and complete the proof in Step~3.

Define \(r_\theta\) to be the dimension of the span (over \(\R\), say) of the vectors in \(R_\theta\cap \Z^d\). The hypotheses of the theorem imply that we cannot have \(r_\theta =d\), since then \(N_\theta\lesssim \delta (\sqrt{\lambda/\delta})^{d-1}\). 


\subsubsection*{Step 1} We show that  there is a basis \(\vecs{1}{x},\dotsc,\vecs{d}{x}\) of \(\Z^d\) with
\begin{equation}\label{eqn:step1}
\delta (\sqrt{\delta\lambda})^{d-1}
\sim
\prod_{i=1}^{d} |\vecs{i}{x}|_\theta^{-1},
\qquad
\#(\Z^d\cap R_\theta)
\sim
\prod_{i=1}^{r_\theta} |\vecs{i}{x}|_\theta^{-1}.
\end{equation}
This is more or less a standard result from the geometry of numbers.

In this step only, let \(A\) be the matrix with \(AR_\theta=[-1,1]^d\), and for \(1\leq i \leq d\) let \(M_i\) be minimal such that there are \(i\) linearly independent  vectors \(\vec{x}\in\Z^d\)  with \(|\vec{x}|_\theta\leq M_i\), or equivalently there are \(i\) linearly independent vectors \(\vec{y}\in A\Z^d\) with \(|A\vec{y}|_\infty \leq M_i\). In particular
\[
M_{r_\theta}\leq 1<M_{r_\theta+1}.
\]
By Theorem~V in \S VIII.4.3 of Cassels~\cite{casselsIntroduction} we have
\begin{equation}\label{eqn:minkowski}
\prod_{i=1}^d M_i
\sim \det A
\sim \frac{1}{\delta(\sqrt{\delta\lambda})^{d-1}}.
\end{equation}
Also by Lemma~2 in \S VIII.1.2 of Cassels   there is a basis  \(\vecs{1}{y},\dotsc,\vecs{d}{y}\) of \(A\Z^d\) such that
\begin{equation}\label{eqn:nice-basis}
M_{i}\leq|\vec{y}|_\infty < M_{i+1},\vec{y}\in A\Z^d
\implies
\vec{y} \in \Z \vecs{1}{y}+\dotsb+\Z \vecs{i}{y}.
\end{equation}
Let \(\vecs{i}{z}\) be the dual basis, so that
\[
\vec{y}
= (\vec{y}\cdot\vecs{1}{z} )\vecs{1}{y}+\dotsb
+(\vec{y}\cdot\vecs{d}{z} )\vecs{d}{y}
\]
for all \(\vec{y}\in A\Z^d\). Note that we must have \(
M_i\leq
|\vecs{i}{y}|_\infty\) from the definition of \(M_i\). By the Corollary to Theorem VIII in \S VIII.5.2 of Cassels, we may choose the \(\vecs{i}{y}\) such that
\begin{equation}\label{eqn:minkowski-basis}
M_i\leq
|\vecs{i}{y}|_\infty
\leq \max\{1,i/2\} M_i,
\qquad
|\vecs{i}{z}|
\leq
(\tfrac{1}{2})^{n-1}(n!)^2/
|\vecs{i}{y}|_\infty.
\end{equation}
We now let \(\vecs{i}{x}=A^{-1}\vecs{i}{y}\) be our basis of \(\Z^d\). We then have
\begin{equation}\label{eqn:basis-norms}
|\vecs{i}{x}|_\theta
=|\vecs{i}{y}|_\infty\sim M_i^{-1},
\end{equation}
by \eqref{eqn:minkowski-basis}. Now \eqref{eqn:basis-norms} and \eqref{eqn:minkowski} together imply the first part of \eqref{eqn:step1}. It also follows from \eqref{eqn:minkowski-basis} that for some \(c>0\) depending only on \(d\) we have
\begin{align*}
|c_i|\leq c M_i^{-1}\;(1\leq i\leq r_\theta)
&\implies
|c_1\vecs{1}{x}+\dotsb+c_{r_\theta} \vecs{r_\theta}{x}|\leq 1,
\intertext{and combining \eqref{eqn:nice-basis} with \eqref{eqn:minkowski-basis} shows that there is a constant \(C>0\) depending only on \(d\) such that}
|c_1\vecs{1}{x}+\dotsb+c_d \vecs{d}{x}|\leq 1
&\implies
|c_i|\leq CM_i^{-1}\;(1\leq i\leq r_\theta),\;
c_i=0\;(i>r_\theta ).
\end{align*}
The last two displays yield
\begin{gather*}
	\#(\Z^d\cap R_\theta)
	=
	\#([-1,1]^d\cap A\Z^d)
	\sim
	\frac{ 1}{\prod_{i=1}^{r_\theta} M_i}
	\sim\frac{ 1}{\prod_{i=1}^{r_\theta} |\vecs{i}{y}|_\infty},
\end{gather*}
and the last part of \eqref{eqn:step1} follows by \eqref{eqn:basis-norms}.

\subsubsection*{Step 2} 
We claim that if \(1\leq r_\theta<d\) and \(\vec{n}_\theta\) is the normal vector to \(\theta\) then there is \(\vec{v}\in\Z^d\setminus\{\vec{0}\}\) such that

\[
|\vec{v}|_2
\lesssim
\delta^{-1}
\left(
\frac{
	\delta(\sqrt{\lambda\delta})^{d-1}
}{
	\#(\Z^d\cap R_\theta)
}
\right)^{
	\frac{1}{d-r_\theta},
}
\qquad
\left| \vec{n}_\theta- \vec{v}/|\vec{v}|_2 \right| \lesssim 
(\sqrt{\lambda\delta})^{-1}
|\vec{v}|_2 ^{-1}
\left(
\frac{
	\delta(\sqrt{\lambda\delta})^{d-1}
}{
	\#(\Z^d\cap R_\theta)
}
\right)^{
	\frac{1}{d-r_\theta}
}.
\]

For the proof, let \(\vecs{i}{x}\) be as in Step~1. Recall the notation  \((\vecs{1}{v}|\cdots|\vecs{k}{v})\) for the matrix with columns \(\vecs{i}{v}\).
 We let
\[
\vec{v}
= \vecs{1}{x}\wedge\dotsb\wedge \vecs{d-1}{x}.
\]
By \eqref{eqn:theta-norm-normal}
we have
\begin{equation}\label{eqn:dirn-of-normal-I}
	\bigg\lvert
	\vec{n}_\theta^T\left(\frac{\vecs{1}{x}}{|\vecs{1}{x}|_\theta} \,\middle|\, \cdots \,\middle|\, \frac{\vecs{d-1}{x}}{|\vecs{d-1}{x}|_\theta}\right) \bigg\rvert_\infty
	\lesssim \delta.
\end{equation}
Now every \(\vecs{i}{x}\) satisfies \(   \frac{\lvert \vecs{1}{x}\rvert_\infty }{|\vecs{1}{x}|_\theta}\lesssim \sqrt{\lambda\delta}\) by \eqref{eqn:theta-norm}. Therefore
\begin{equation}\label{eqn:dirn-of-normal-II}
	\left(\frac{\vecs{1}{x}}{|\vecs{1}{x}|_\theta} \,\middle|\, \cdots \,\middle|\, \frac{\vecs{d-1}{x}}{|\vecs{d-1}{x}|_\theta}
	\right) = U
	\begin{pmatrix}
		\operatorname{diag}(d_1,\dotsc,d_{d-1})
		\\O_{1\times (d-1)}
	\end{pmatrix}
	V,
\end{equation}
where  \(O\) is a matrix of zeroes, \(U,V\) are orthogonal, and \(d_i\lesssim \sqrt{\lambda\delta}\). It follows that
\begin{equation}\label{eqn:size-of-v}
|\vec{v}|_2
\sim
\prod_{i=1}^{d-1}
d_i |\vecs{i}{x}|_\theta
\lesssim
\delta^{-1}|\vecs{d}{x}|^{-1}_\theta.
\end{equation}
Now \eqref{eqn:dirn-of-normal-I} and \eqref{eqn:dirn-of-normal-II} together yield
\[
\bigg\lvert\vec{n}_\theta^T U
\left(\begin{smallmatrix}
	1\vspace{-.5em}\\&\ddots\\&&1\\&&&0
\end{smallmatrix}\right)\bigg\rvert_\infty
\lesssim 
\delta
(\min_i |d_i|)^{-1}.
\]
For any vector with \(|\vec{m}|_2=	1\) we have
\[
|m_d|-1
\lesssim
\bigg\lvert\vec{m}^T
\left(\begin{smallmatrix}
	1\vspace{-.5em}\\&\ddots\\&&1\\&&&0
\end{smallmatrix}\right)\bigg\rvert_2^2,
\]
and so
\[
\min\{
|\vec{m}-\vecs{d}{e}|_2,
|\vec{m}+\vecs{d}{e}|_2
\}
\lesssim\bigg\lvert\vec{m}^T
\left(\begin{smallmatrix}
	1\vspace{-.5em}\\&\ddots\\&&1\\&&&0
\end{smallmatrix}\right)\bigg\rvert_\infty.
\]
It follows that for some choice of sign,
\[
|\vec{n}_\theta\pm U\vecs{d}{e}|_2
=
|U^T\vec{n}_\theta\pm \vecs{d}{e}|_2
\lesssim 
\delta (\min_i |d_i|)^{-1}.
\]
By our definititon \(\vec{v}= \vecs{1}{x}\wedge\dotsb\wedge \vecs{d-1}{x}\) we have
\begin{gather*}
	\vec{v}^T
\left(\frac{\vecs{1}{x}}{|\vecs{1}{x}|_\theta} \,\middle|\, \cdots \,\middle|\, \frac{\vecs{d-1}{x}}{|\vecs{d-1}{x}|_\theta}
\right)
=(0,\dotsc,0),
\end{gather*}
so that \(U\vecs{d}{e}=\vec{v}/|\vec{v}|_2\).  
It follows that, possibly after replacing \(\vec{v}\) with \(-\vec{v}\) if necessary, we have
\begin{equation}
	\label{eqn:approx_by_v}
\left| \vec{n}_\theta- \vec{v}/|\vec{v}|_2 \right| \lesssim 
\frac{
	\delta(\sqrt{\lambda\delta})^{d-2}
	\prod_{i=1}^{d-1} |\vecs{i}{x}|_\theta
}{
|\vec{v}|_2 
}
\sim
\frac{
	1
}{
	\sqrt{\lambda\delta}
	|\vecs{d}{x}|_\theta
	|\vec{v}|_2 
},
\end{equation}
where the last part follows by Step~1. It remains to observe that, again by Step~1, we have
\[
|\vecs{d}{x}|_\theta^{d-r_\theta}
\geq
\prod_{i=r_\theta+1}^{d}
|\vecs{i}{x}|_\theta
\sim
\frac{\#(\Z^d\cap R_\theta)}{\delta (\sqrt{\delta\lambda})^{d-1}},
\]
\textcolor{black}{and the result follows from this bound together with \eqref{eqn:size-of-v} and \eqref{eqn:approx_by_v}.}

\subsubsection*{Step 3}

By Step~2, the number of \(\theta \in \mathcal{C}\) with \(r_\theta=r\) and \(\#(\Z^d\cap R_\theta) \sim 2^J\delta (\sqrt{\delta\lambda})^{d-1}\) is 
\[
\lesssim
\sum_{ \substack{\vec{v}\in \Z^d\setminus\{\vec{0}\}\\ |\vec{v}|_2 \leq 
\delta^{-1} 2^{
	{-J/(d-r)}
}}}
\left(1+
\delta^{-1}
|\vec{v}|_2 ^{-1}
2^{-J/(d-r)
}
\right)^{d-1}
\lesssim
\left(
\delta 2^{J/(d-r)}
\right)^{-d}.
\]

Let \(2^j\) be as in the theorem and suppose \(\theta\in \mathcal{C}_j\). Then \[\#(\Z^d\cap R_\theta)\geq N_\theta >2^{j-1}\delta (\sqrt{\lambda\delta})^{d-1},\] and it follows that \(\#(\Z^d\cap R_\theta)>{\color{black}{\frac{1}{2}}}K(\sqrt{\delta\lambda})^{d-k-1}\).

\textcolor{black}{Recall \(R_\theta\cap \Z^d\) is contained in a box of size $\sim \delta \times \sqrt{\lambda \delta} \times \dots \times \sqrt{\lambda \delta}$, intersected with a linear space of dimension \(r_\theta<d\), and so
		\begin{equation*}
				\label{eqn:triv_upper_bound}
				\#(\Z^d\cap R_\theta)
				\ll
				\sqrt{\delta\lambda}^{r_\theta}.
			\end{equation*}
Hence \(r_\theta\geq d-k\) since \(K\) may be taken arbitrarily large. So the number of possible  \(\theta \) is \(\lesssim  (\delta 2^{j/k})^{-d}\).}
\end{proof}

\section{Application of the $\ell^2$ decoupling theorem}

We decompose $P_{\lambda,\delta}$ into projectors $P_{\lambda,\delta}^j$ which are supported on the union of caps in $\mathcal{C}_j$. To be more specific, we choose a partition of unity $(\chi_\theta)$ adapted to the caps defined in the previous section:
$$
\operatorname{Supp} \chi_\theta \subset \theta \qquad \mbox{and} \qquad \sum_\theta \chi_\theta = 1 \quad \mbox{on $S_{\lambda, \delta}$},
$$
and let
\begin{align*}
& \chi_j = \sum_{\theta \in \mathcal{C}_j} \chi_\theta \\
& P_{\lambda,\delta}^j = P_{\lambda,\delta}\, \chi_j(D)
\end{align*}
(where $\chi_j(D)$ is the Fourier multiplier with symbol $\chi_j(k)$).
Using the $\ell^2$ decoupling theorem of Bourgain and Demeter~\cite{BourgainDemeter3}, we can estimate the operator norm of $P_{\lambda,\delta}^j$ from $L^2$ to $L^{p}$:

\begin{prop} \label{propj}
For any $Q$, for any $\epsilon>0$, for $p \geq p_{ST}$, and for $\delta > \lambda^{-1}$,
$$
\| P_{\lambda,\delta}^j \|_{L^{2} \to L^{p}} \lesssim_\epsilon \lambda^{\frac{\sigma(p)}{2} +\epsilon} \delta^{1/2} 2^{j \left( \frac{1}{2} - \frac{1}{p} \right)}.
$$
\end{prop}

\begin{proof} For simplicity in the notation, we only consider the case $Q = \operatorname{Id}$. 
	Let $a_k$  be an arbitrary sequence in $\ell^2(\mathbb{Z}^d)$, or in other words the Fourier series associated to an arbitrary function in $L^2(\mathbb{T}^d)$.
	Changing variables to $X = \lambda x$ and $K = k/ \lambda$, and taking advantage of the periodicity of Fourier series, we get
\begin{align*}
& \left\| \sum_{k \in \mathbb{Z}^d} \chi_j(k) \chi\left( \frac{|k|-\lambda}{\delta} \right) a_k e^{2\pi i k\cdot x} \right\|_{L^p(\mathbb{T}^d)} \\
& \qquad \qquad \lesssim
\left( \frac{\delta}{\lambda} \right)^{d/p} \left\| \phi \left( \frac{\delta X}{\lambda} \right) \sum_{K \in \mathbb{Z}^d/\lambda} \chi_j(\lambda K)\chi\left( \frac{|K|-1}{(\delta/\lambda)} \right) a_{\lambda K} e^{2\pi i K \cdot X} \right\|_{L^{p}(\mathbb{R}^d)},
\end{align*}
where the cutoff function $\phi$ can be chosen to have compactly supported Fourier transform. As a result, the Fourier transform of the function on the right-hand side is supported on a $\delta/\lambda$-neighborhood of $\mathbb{S}^{d-1}$.
Using the $\ell^2$ decoupling theorem of Bourgain and Demeter
\footnote{The theorem, as stated in that paper, does not immediately apply to our setup. The following procedure can be applied: first, restrict to a coordinate patch on the sphere. Second, split our caps $\theta$ into a finite number of subcollections $\mathcal{S}_j$, with the following property: for each $j$, there exists a covering $\mathcal{P}_j$ as in Bourgain-Demeter such that any $\theta \in \mathcal{S}_j$ is contained in one element in $\mathcal{P}_j$. Third, sum over $j$ to obtain the result. Alternatively, one can resort to the version in Tao ~\cite{Tao}, Exercise 26.},
this is
$$
\dots \lesssim_\epsilon \left( \frac{\delta}{\lambda} \right)^{\frac{d}{p}- \frac{d-1}{4} + \frac{d+1}{2p} - \epsilon} \left( \sum_{\theta \in \mathcal{C}_j} \left\|  \phi \left( \frac{\delta X}{\lambda} \right)  \sum_{K \in \mathbb{Z}^d / \lambda} \chi_\theta(\lambda K) \chi \left( \frac{|K|-1}{(\delta/\lambda)} \right) a_{\lambda K} e^{2\pi i K \cdot X}\right\|_{L^p(\mathbb{R}^d)}^2  \right)^{1/2}
$$
(notice that $\theta/\lambda$ has dimensions  $\sim \frac{\delta}{\lambda} \times \frac{\delta^{1/2}}{\lambda^{1/2}} \dots \times \frac{\delta^{1/2}}{\lambda^{1/2}}$). At this point, we use the inequality
$$
\mbox{if $p \geq 2$}, \qquad \| f \|_{L^p(\mathbb{R}^d)} \lesssim \| f \|_{L^2} | \operatorname{Supp} \widehat{f} |^{\frac{1}{2} - \frac{1}{p}},
$$
which follows by applying successively the Hausdorff-Young and H\"older inequalities, and finally the Plancherel equality. We use this inequality for $f =  \phi \left( \frac{\delta X}{\lambda} \right)  \sum_{K} \chi \left( \frac{|K|-1}{(\delta/\lambda)} \right) \chi_\theta(\lambda K) a_{\lambda K} e^{2\pi i K \cdot X}$. Since $\theta \in \mathcal{C}_j$, its Fourier transform is supported on the union of at most $O((\delta \lambda)^{\frac{d-1}{2}}\delta 2^j)$ balls of radius $O( \delta / \lambda) $, giving $| \operatorname{Supp} \widehat{f} | \lesssim \delta^{\frac{3d+1}{2}} \lambda^{-\frac{d+1}{2}} 2^j$. Coming back to the quantity we want to bound, it is
\begin{align*}
 & \lesssim  \left( \frac{\delta}{\lambda} \right)^{\frac{d}{p}- \frac{d-1}{4} + \frac{d+1}{2p} - \epsilon} \left(\delta^{\frac{3d+1}{2}} \lambda^{-\frac{d+1}{2}} 2^j \right)^{\frac{1}{2} - \frac{1}{p}} \\ & \qquad \qquad \qquad \qquad \left( \sum_{\theta \in \mathcal{C}_j} \left\|  \phi \left( \frac{\delta X}{\lambda} \right)  \sum_{K \in \mathbb{Z}^d / \lambda} \chi_\theta(\lambda K) \chi \left( \frac{|K|-1}{(\delta/\lambda)} \right) a_{\lambda K} e^{2\pi i K \cdot X}\right\|_{L^2(\mathbb{R}^d)}^2  \right)^{1/2} \\
& \lesssim \left( \frac{\delta}{\lambda} \right)^{\frac{d}{p} - \frac{d-1}{4} + \frac{d+1}{2p} - \epsilon}\left( \delta^{\frac{3d+1}{2}} \lambda^{-\frac{d+1}{2}} 2^j\right)^{\frac{1}{2} - \frac{1}{p}} \left\|  \phi \left( \frac{\delta X}{\lambda} \right) \sum_{K \in \mathbb{Z}^d/\lambda} \chi_j(\lambda K)\chi\left( \frac{|K|-1}{(\delta/\lambda)} \right) a_{\lambda K} e^{2\pi i K \cdot X} \right\|_{L^2(\mathbb{R}^d)},
\end{align*}
where the last inequality is a consequence of almost orthogonality. Finally undoing the change of variables, this is 
\begin{align*}
& \lesssim \left( \frac{\delta}{\lambda} \right)^{\frac{d}{p}-\frac{d}{2} - \frac{d-1}{4} + \frac{d+1}{2p} - \epsilon}
\left( \delta^{\frac{3d+1}{2}} \lambda^{-\frac{d+1}{2}} 2^j \right)^{\frac{1}{2} - \frac{1}{p}}
 \left\| \sum_{k \in \mathbb{Z}^d} \chi_j(k) \chi\left( \frac{|k|-\lambda}{\delta} \right) a_k e^{2\pi i k\cdot x} \right\|_{L^2(\mathbb{T}^d)} \\
& \leq  \left( \frac{\delta}{\lambda} \right)^{-\epsilon} \lambda^{\sigma(p)/2} \delta^{1/2} 2^{j \left( \frac{1}{2} - \frac{1}{p} \right)} \left\| \sum_{k \in \mathbb{Z}^d} a_k e^{2\pi i k\cdot x} \right\|_{L^2(\mathbb{T}^d)}.
\end{align*}
\end{proof}

\section{Proof of the main theorems}

\subsection{The case $p<p_{ST}$: proof of Theorem~\ref{thmpST}} Proposition~\ref{propj} gives the bounds
$$
\| P_{\lambda,\delta} \|_{L^2 \to L^{p_{ST}}} \leq \sum_{j=0}^\infty \| P_{\lambda,\delta}^j \|_{L^2 \to L^{p_{ST}}} \lesssim \lambda^{\epsilon} (\delta \lambda)^{\frac{d-1}{2(d+1)}}.
$$
Interpolating with the trivial $L^2 \to L^2$ bound, this gives the conjecture for $2 \leq p \leq p_{ST}$.

\subsection{An exact but involved statement}

The theorems in the introduction are deduced from the following result. \textcolor{black}{The first bound \eqref{eqn:full_bound} in this next theorem is precisely the result of  interpolating between the bounds obtained above.} The last part of the theorem allows for the concise result in Theorem~\ref{thmsimple} but, as we will see in the proof, it is slightly weaker.

\begin{thm}\label{thm:main}	\textcolor{black}{
	Assume $p_{ST} \leq p \leq \infty$ and write
	\begin{align}\label{eqn:excess-of-p}
		\alpha(p)
		&=1-\frac{2}{p},
		&
		\beta(p)
		&=
		1- \frac{p_{ST}}{p}.
	\end{align}
	Assume further that $\delta> \lambda^{-\frac{d-1}{d+1}}$. Then
\begin{multline}\label{eqn:full_bound}
	\| P_{\lambda,\delta} \|_{L^2 \to L^p}
	\lesssim \lambda^{\sigma(p)/2} \delta^{1/2} + 
	\lambda^\epsilon
	(\lambda\delta)^{\frac{d}{4}\alpha(p)}
	( \lambda/\delta)^{\frac{d}{4}\beta(p)}
	\sum_{\substack{ 1\leq k\leq d-1 \\ (\delta \lambda)^{(k-1)/2} \delta<1 }}
	(\delta \lambda)^{-\frac{k}{4}\alpha(p)}
	( \lambda/\delta)^{\frac{-1}{2k}\frac{d}{2}\beta(p)}
%
\\
+
\lambda^\epsilon
(\lambda\delta)^{\frac{d}{4}\alpha({p})}(\lambda/\delta)^{-\frac{1}{4}\alpha({p})} \delta^{-\frac{d}{2}\beta({p})}.
\end{multline}
	and it follows that
	\begin{multline*}
	\| P_{\lambda,\delta} \|_{L^2 \to L^p}
	\lesssim 
	\lambda^{\sigma(p)/2} \delta^{1/2} 
	+ 
	\lambda^\epsilon
	(\lambda\delta)^{\frac{d}{4}\alpha({p})}(\lambda/\delta)^{-\frac{1}{4}\alpha({p})} \delta^{-\frac{d}{2}\beta({p})}
	\\
	+ 
	\lambda^\epsilon
	(\lambda\delta)^{\frac{d}{4}\alpha(p)}
	( \lambda/\delta)^{\frac{d}{4}\beta(p)}
	e^{-\frac{1}{2}\sqrt{d\alpha(p)
			\beta(p)\log (\delta \lambda)\log ( \lambda/\delta)}},
	\end{multline*}
where the final term may be omitted if \(\delta > 
\lambda^{\frac{\alpha(p)-d\beta(p)}{\alpha(p)+d\beta(p)}}\).
}
\end{thm}

In order to prove the theorem, we need a brief lemma.

\begin{lem}\label{lem:optimalk}
We introduce the notation $k_0$ for the optimal index in Theorem~\ref{thmcaps} given $\delta,\lambda, 2^j$. Namely, assume that $\delta> \lambda^{-\frac{d-1}{d+1}}$. Then, let $k_0 = k_0(\delta,\lambda,2^j)$ be the smallest $k\in\Z$ such that
	$$
	(\delta \lambda)^{k/2} \delta 2^j >K.
	$$
	If \(1< 2^j<K\delta^{-1}\) then \(k_0(\delta,\lambda,2^j)\in \{ 1 ,\dots,d-1\}\). 
\end{lem}
\begin{proof}
Let \(1< 2^j<K\delta^{-1}\), then
\[
K(\delta \lambda)^{-k/2} \delta^{-1}
<2^j
\leq K(\delta \lambda)^{(1-k)/2} \delta^{-1} 
\iff
k_0=k,
\]
and so it suffices to observe that as $\delta> \lambda^{-\frac{d-1}{d+1}}$ we have \((\delta \lambda)^{-(d-1)/2} \delta^{-1}< 1\). 
\end{proof}

\begin{proof}[Proof of Theorem~\ref{thm:main}]
	Throughout the proof, \(k_0\) will be as in Lemma~\ref{lem:optimalk}.
\textcolor{black}{We begin by bounding $P^0_{\lambda,\delta}$, by interpolating between}
\begin{align*}
\| P^0_{\lambda,\delta} \|_{L^2 \to L^{p_{ST}}} \lesssim \lambda^{\frac{d-1}{2(d+1)}} \delta^{1/2},
\end{align*}
which is a consequence of Proposition~\ref{propj}, and
\begin{align*}
\| P^0_{\lambda,\delta} \|_{L^2 \to L^\infty} \lesssim 
{\bigg\lvert\bigcup_{\theta \in \mathcal{C}_0} \Z^d\cap \theta\bigg\rvert^{1/2}}
\leq
\sqrt{|\mathcal{C}|(\sqrt{\lambda \delta})^{d-1} \delta }
\lesssim
(\lambda^{d-1} \delta)^{1/2}.
\end{align*}
Interpolating between these two estimates gives
$$
\| P^0_{\lambda,\delta} \|_{L^2 \to L^p} \lesssim \lambda^{\sigma(p)/2} \delta^{1/2}.
$$

If \(2^j\geq K\delta^{-1}\) then we can assume \(\mathcal{C}_j=\emptyset\) by \eqref{eqn:C-max}, by increasing the size of the constant \(K\) if neccessary. It follows that \(P^j_{\lambda,\delta}=0\) for such \(j\).

Next let \(1< 2^j<K\delta^{-1}\).  Now on the one hand, Proposition~\ref{propj} gives
$$
\| P^j_{\lambda,\delta} \|_{L^2 \to L^{p_{ST}}} \lesssim \lambda^{\frac{1}{p_{ST}}} \delta^{1/2} 2^{\frac{j}{d+1}}.
$$
\textcolor{black}{On the other hand}, bounding the number of points in $\cup_{\mathcal{C}_j} \theta$ through  Lemma~\ref{lem:optimalk} and Theorem~\ref{thmcaps} gives
$$
\| P^j_{\lambda,\delta}\|_{L^2 \to L^\infty} \lesssim \left( \# \mathcal{C}_j (\lambda \delta)^{\frac{d-1}{2}} \delta 2^j \right)^{1/2} \lesssim \left( (2^{j/k_0} \delta)^{-d} (\lambda \delta)^{\frac{d-1}{2}} \delta 2^j \right)^{1/2}.
$$
Interpolating between the last two bounds 	gives
\begin{align*}
\| P^j_{\lambda,\delta}\|_{L^2 \to L^p}
&\lesssim
 \lambda^{\frac{d-1}{4} - \frac{d-1}{2p}} \delta^{\frac{d+1}{4} - \frac{d+1}{2p}} 2^{\frac{j}{2} - \frac{j}{p}} (2^{j/k_0} \delta)^{-\frac{d}{2} + \frac{d(d+1)}{d-1} \frac{1}{p}}
 \\
 &=
 (\lambda\delta)^{\frac{d}{4}(1- \frac{2}{p})}(\lambda/\delta)^{-\frac{1}{4}(1- \frac{2}{p})}
 2^{\frac{j}{2}(1-  \frac{2}{p})} (2^{j/k_0} \delta)^{-\frac{d}{2}(1- \frac{p_{ST}}{p})}.
\end{align*}
\textcolor{black}{From this point on it will simplify matters to write \(\alpha(p)=1-\frac{2}{p},\beta(p)=1-\frac{p_{ST}}{p}\) as in \eqref{eqn:excess-of-p}.}
On
 summing over $j$, we obtain
 	\begin{equation}\label{eqn:main_first_version}
 		\| P_{\lambda,\delta} \|_{L^2 \to L^p} \lesssim
 		 \lambda^{\sigma(p)/2} \delta^{1/2} +
 		 (\lambda\delta)^{\frac{d}{4}\alpha({p})}(\lambda/\delta)^{-\frac{1}{4}\alpha({p})}
 		 \sum_{1< 2^j < K\delta^{-1}}
 		 2^{\frac{j}{2}\alpha({p})} (2^{j/k_0}\delta )^{-\frac{d}{2}\beta({p})}.
 \end{equation}\textcolor{black}{
 Recall that \(k_0\) is minimal such that \((\delta \lambda)^{-k/2} \delta^{-1}
 <2^j \), and that \(1\leq k_0\leq d-1\) by Lemma~\ref{lem:optimalk}. Hence
\begin{align*}
\sum_{1< 2^j < K\delta^{-1}}
2^{\frac{j}{2}\alpha({p})} (2^{j/k_0}\delta )^{-\frac{d}{2}\beta({p})}
%
&\leq
\sum_{1\leq k\leq d-1}
\sum_{\substack{j\in\N \\ (\delta \lambda)^{-k/2} \delta^{-1}
	<2^j \\ (\delta \lambda)^{(1-k)/2} \delta^{-1} \geq 2^j
\\
1< 2^j < K\delta^{-1}
}}
 2^{\frac{j}{2}\alpha({p})} (2^{j/k} \delta)^{-\frac{d}{2}\beta({p})}
 \\
& \lesssim_\epsilon 
 \sum_{1\leq k\leq d-1}\lambda^\epsilon
 \sum_{\substack{j \in\R \\ 2^j=\max\{1,(\delta \lambda)^{-\ell/2} \delta^{-1}\}
 		\\\ell\in\{k-1,k\}}}
 2^{\frac{j}{2}\alpha({p})} (2^{j/k} \delta)^{-\frac{d}{2}\beta({p})},
 \intertext{and exchanging order of summation this is}
 & =
 \lambda^\epsilon
 \sum_{\substack{j \in\R \\ 2^j=\max\{1,(\delta \lambda)^{-\ell/2} \delta^{-1}\}
 		\\0\leq \ell\leq d-1}}
 	\sum_{\substack{ k\in\{\ell,\ell+1\} \\ 1\leq k\leq d-1}}
 2^{\frac{j}{2}\alpha({p})} (2^{j/k} \delta)^{-\frac{d}{2}\beta({p})}.
%
%
\end{align*}
On recalling that \(p\geq p_{ST}\) and so \(\beta(p)\geq 0\), we see  that when \(\ell\leq d-2\), we can obtain an upper bound for the inner sum in the last line above by substituting \(k= \ell+1\). Thus
\begin{align*}
	\sum_{1< 2^j < K\delta^{-1}}
	2^{\frac{j}{2}\alpha({p})} (2^{j/k_0}\delta )^{-\frac{d}{2}\beta({p})}
&\lesssim_{\epsilon}
\lambda^\epsilon
\sum_{\substack{j \in\R \\ 2^j=\max\{1,(\delta \lambda)^{-\ell/2} \delta^{-1}\}	\\0\leq \ell\leq d-2}}
2^{\frac{j}{2}\alpha({p})} (2^{\frac{j}{\ell+1}} \delta)^{-\frac{d}{2}\beta({p})}
\\
&\hphantom{{}\leq{}}+
\lambda^\epsilon
\max\left\{1,(\delta \lambda)^{-(d-1)/2} \delta^{-1}\right\}^{\frac{1}{2}\alpha({p})-\frac{d}{2(d-1)}\beta({p})}  \delta^{-\frac{d}{2}\beta({p})}
%
\\
&=
\lambda^\epsilon
\sum_{\substack{ 1\leq k\leq d-1 \\ (\delta \lambda)^{(k-1)/2} \delta<1 }}
\left((\delta \lambda)^{-(k-1)/2} \delta^{-1}\right)^{\frac{1}{2}\alpha({p})-\frac{d}{2k}\beta({p})}  \delta^{-\frac{d}{2}\beta({p})}
+
\lambda^\epsilon \delta^{-\frac{d}{2}\beta({p})},
\end{align*}
where to deal with the maximum we use the assumption \(\delta>\lambda^{-\frac{d-1}{d+1}}\) from the theorem.Upon writing \((\delta \lambda)^{-(k-1)/2} \delta^{-1} = (\delta\lambda)^{-k/2}  (\lambda/\delta)^{1/2}\), it now follows by  \eqref{eqn:main_first_version} that
\begin{multline*}
\| P_{\lambda,\delta} \|_{L^2 \to L^p} 
 \lesssim
 \lambda^{\sigma(p)/2} \delta^{1/2} + 
 \\
\lambda^\epsilon
(\lambda\delta)^{\frac{d}{4}\alpha({p})}(\lambda/\delta)^{-\frac{1}{4}\alpha({p})}
\sum_{\substack{ 1\leq k\leq d-1 \\ (\delta \lambda)^{(k-1)/2} \delta<1 }}
(\delta\lambda)^{-\frac{k}{2}\big(\frac{1}{2}\alpha({p})-\frac{d}{2k}\beta({p})\big)-\frac{d}{4}\beta(p)}
(\lambda/\delta)^{\frac{1}{2}\big(\frac{1}{2}\alpha({p})-\frac{d}{2k}\beta({p})\big)+\frac{d}{4}\beta(p)}
%
\\
+
\lambda^\epsilon
(\lambda\delta)^{\frac{d}{4}\alpha({p})}(\lambda/\delta)^{-\frac{1}{4}\alpha({p})} \delta^{-\frac{d}{2}\beta({p})}.
\end{multline*}
This proves \eqref{eqn:full_bound}, and we now proceed to deduce the last part of the theorem. If \(A,B>0\) then \[\exp(-kA-B/k)\leq \exp(-2\sqrt{AB}),\] and  moreover 
\[
\sum_{\substack{ 1\leq k\leq d-1 \\ k < k_1 }}
\exp(-kA-B/k)\lesssim
	\exp(-k_1A-B/k_1) \qquad\text{if }
	k_1\leq \sqrt{B/A}.
\]
We apply this with
\begin{align*}
4A&=\alpha(p)\log (\delta \lambda)
,&
4B&=d\beta(p)\log ( \lambda/\delta)
&
k_1 &= \log (\lambda/\delta) /\log (\lambda\delta),
\end{align*}
so that \(
(\delta \lambda)^{(k_1-1)/2} \delta=1\).
Noting that
\[
\exp(
-k_1A-B/k_1)
=
\exp\Big(
-
\tfrac{1}{4}\alpha(p)\log(\lambda/\delta)
-
\tfrac{d}{4}\beta(p) \log (\delta \lambda)\Big)
=
\delta^{-\frac{d}{2}\beta(p)}(\lambda/\delta)^{-\frac{d}{4} \beta(p)-\frac{1}{4}\alpha(p) }
,
\]
we deduce from \eqref{eqn:full_bound} that
\begin{multline*}
	\| P_{\lambda,\delta} \|_{L^2 \to L^p}
	\lesssim \lambda^{\sigma(p)/2} \delta^{1/2} 
	+
	\lambda^\epsilon
	(\lambda\delta)^{\frac{d}{4}\alpha({p})}(\lambda/\delta)^{-\frac{1}{4}\alpha({p})} \delta^{-\frac{d}{2}\beta({p})}
	\\
	+ 
	\lambda^\epsilon
	(\lambda\delta)^{\frac{d}{4}\alpha(p)}
	( \lambda/\delta)^{\frac{d}{4}\beta(p)}
	e^{-\frac{1}{2}\sqrt{d\alpha(p)
			\beta(p)\log (\delta \lambda)\log ( \lambda/\delta)}},
\end{multline*}
where the last term is omitted if \(k_1 <\sqrt{B/A}\).
This is the remaining bound in Theorem~\ref{thm:main}.
}
\end{proof}

\subsection{From Theorem~\ref{thm:main} to Theorems~\ref{thmsimple} and~\ref{thmd3}}

\begin{proof}[Proof of Theorem~\ref{thmsimple}]
	\textcolor{black}{We use throughout the proof the notation  \(\alpha(p)=1-\frac{2}{p},\beta(p)=1-\frac{p_{ST}}{p}\) from \eqref{eqn:excess-of-p}. As in theorem we assume  \(\delta > 
	\lambda^{\frac{\alpha(p)-d\beta(p)}{\alpha(p)+d\beta(p)}}\). We claim that \(\delta\geq \lambda^{-\frac{d-1}{d+1}}\), that is to say
	\[
	\frac{1-\beta(p)\alpha(p)^{-1}d}{1-\beta(p)\alpha(p)^{-1}d}>
	\frac{1-d}{1+d},
	\]
	which is true since \(\beta(p)\alpha(p)^{-1} < 1\) and \(\frac{1-dx}{1+dx}\) is an increasing function of \(x\).}
	
The last bound in Theorem~\ref{thm:main} will now agree with Conjecture~\ref{conj} if
\textcolor{black}{
\begin{align}
(\lambda\delta)^{\frac{d}{4}\alpha({p})}(\lambda/\delta)^{-\frac{1}{4}\alpha({p})} \delta^{-\frac{d}{2}\beta({p})}
&\leq\lambda^{\frac{d-1}{2}-\frac{d}{p}} \delta^{1/2}. 
\label{eqn:easy-case}
\end{align}
Grouping all the terms with a \(1/p\) and without a \(1/p\) in their exponent, the bound \eqref{eqn:easy-case} is exactly
\[
\lambda^{
\frac{d+1}{2p}
}
\delta^{\frac{(d+1)^2}{2(d-1)p}}
\leq 
\lambda^{\frac{d-1}{4}}
\delta^{\frac{d+1}{4}},
\]
which is to say \((\lambda \delta^{\frac{d+1}{d-1}})^{-\frac{d-1}{4}\beta(p)}\leq 1 \), and this is true since \(\delta\geq \lambda^{-\frac{d-1}{d+1}}\) and \(p>p_{ST}\).}
\end{proof}

\begin{proof}[Proof of Theorem~\ref{thmd3}]
	\textcolor{black}{As in the last proof, we use throughout the notation  \(\alpha(p)=1-\frac{2}{p},\beta(p)=1-\frac{p_{ST}}{p}\) from \eqref{eqn:excess-of-p}, noting that since \(d=3\) we have \(p_{ST}=4\) and \(\beta(p)=1-\frac{4}{p}\).}
	If \(d=3\)  the first bound in Theorem~\ref{thm:main} agrees with the conjecture when
\begin{multline*}
	(\lambda\delta)^{\frac{3}{4}\alpha(p)}
	( \lambda/\delta)^{\frac{3}{4}\beta(p)}
	(\delta \lambda)^{-\frac{1}{4}\alpha(p)}
	( \lambda/\delta)^{-\frac{3}{4}\beta(p)}
	\\
	+
	(\lambda\delta)^{\frac{3}{4}\alpha(p)}
	( \lambda/\delta)^{\frac{3}{4}\beta(p)}
	(\delta \lambda)^{-\frac{2}{4}\alpha(p)}
	( \lambda/\delta)^{-\frac{3}{8}\beta(p)}
	\leq
	 (\lambda\delta)^{\frac{3}{4}\alpha(p)}
	 (\lambda/\delta)^{\frac{1}{4}-\frac{3}{2p}} +(\lambda\delta)^{\frac{1}{2}\alpha(p)},
\end{multline*}
that is
\begin{equation*}
(\lambda\delta)^{\alpha(p)}
+
(\lambda\delta)^{\frac{1}{4}\alpha(p)}
( \lambda/\delta)^{\frac{3}{8}\beta(p)}
\leq
(\lambda\delta)^{\frac{3}{4}\alpha(p)}
(\lambda/\delta)^{\frac{1}{4}-\frac{3}{2p}} +(\lambda\delta)^{\frac{1}{2}\alpha(p)}.
\end{equation*}
The first term on the left-hand side is clearly bounded by the last term. The second term on the left-hand side is bounded by the right-hand side if
\begin{equation*}
( \lambda/\delta)^{\frac{1}{8}}
\leq 
(\lambda\delta)^{\frac{1}{2}\alpha(p)}
\qquad\text{or}\qquad
( \lambda/\delta)^{\frac{3}{8}\beta(p)}
\leq
(\lambda\delta)^{\frac{1}{4}\alpha(p)}.
\end{equation*}
This is 
\(\delta \geq \min\{\lambda^{-\frac{3p-8}{5p-8}}\, \lambda^{-\frac{8-p}{5p-16}}\}\), and recalling our standing assumption \(\delta\geq \lambda^{-(d-1)/(d+1)}\) from Theorem~\ref{thm:main} this gives the result.
\end{proof}

\appendix

\section{The Euclidean case}

We prove here the estimate~\eqref{swallow} and show its optimality; the arguments are classical and elementary, but we give them here for ease of reference.

First notice the scaling relation $\| P_{\lambda,\delta} \|_{L^2 \to L^p} = \lambda^{\frac{d}{2} - \frac{d}{p}} \| P_{1,\delta/\lambda} \|_{L^2 \to L^p}$, which reduces matters to $\lambda = 1$: it suffices to prove that
$$
\| P_{1,\delta} \|_{L^2 \to L^p} \lesssim
\left\{
\begin{array}{ll}
\delta^{1/2} & \mbox{if $p \geq p_{ST}$} \\
\delta^{\frac{(d+1)}{2}\left( \frac{1}{2} - \frac{1}{p} \right)} & \mbox{if $2 \leq p \leq p_{ST}$},
\end{array}.
\right.
$$
This is achieved by interpolating between the following points:
\begin{itemize}
\item $p=2$, which is trivial by Plancherel's theorem.
\item $p=\infty$, which follows from the Hausdorff-Young and Cauchy-Schwarz inequality, as well as Plancherel's theorem:
$$
\| P_{1, \delta} f \|_{L^\infty} \lesssim \left\| \chi \left( \frac{|\xi| - 1}{\delta} \right) \widehat{f} \right\|_{L^1} \lesssim \left\|  \chi \left( \frac{|\xi| - 1}{\delta} \right)  \right\|_{L^2} \| \widehat f \|_{L^2} \lesssim \delta^{1/2} \| f \|_{L^2}.
$$
\item $p = p_{ST}$, for which we will use the formula
$$
P_{1,\delta} f = \int_0^\infty  \chi \left( \frac{r - 1}{\delta} \right) \int_{\mathbb{R}^d} \widehat{f}(\xi) e^{ix\cdot \xi}\, d\sigma_r(\xi) \,dr,
$$
where $d\sigma_r$ is the surface measure on the sphere $S_r$ with center at the origin and radius $r$. We can then apply successively the Minkowski inequality, the Stein-Tomas theorem~\cite{SteinTomas,Stein}, the Cauchy-Schwarz inequality and the Plancherel theorem to obtain
\begin{align*}
\| P_{1, \delta} f \|_{L^{p_{ST}}(\mathbb{R}^d)} & \leq  \int_0^\infty  \chi \left( \frac{r - 1}{\delta} \right)  \left\|  \int \widehat{f}(\xi) e^{ix\cdot \xi} \, d\sigma_r(\xi) \right\|_{L^{p_{ST}}(\mathbb{R}^d)} \,dr \\
& \leq  \int_0^\infty  \chi \left( \frac{r - 1}{\delta} \right) \| \widehat{f} \|_{L^2(S_r)}\,dr \lesssim \delta^{1/2} \|f\|_{L^2(\mathbb{R}^d)}.
\end{align*}

\end{itemize}
Finally, there remains to check optimality. It follows from two examples:
\begin{itemize}
\item The Knapp example is a function $\widehat{f}$ which is a cutoff function adapted to a rectangular box of size $\sim \delta$ in one direction, and $\sim \delta^{1/2}$ in $d-1$ directions; this box is furthermore chosen to be contained in $B(0,r+\delta) \setminus B(0,r-\delta)$. Such a function is easily seen to achieve
$$
\frac{\| P_{1,\delta} f \|_{L^p}}{\| f \|_{L^2}} \sim \delta^{\frac{(d+1)}{2}\left( \frac{1}{2} - \frac{1}{p} \right)}
$$
\item The radial example is $\widehat{g}(\xi) = \chi \left( \frac{|\xi| - \lambda}{\delta} \right) $. Using the fact that the Fourier transform of the surface measure on the unit sphere decays like $|\xi|^{-\frac{d-1}{2}}$ as $|\xi| \to \infty$, one can check that
$$
\frac{\| P_{1,\delta} f \|_{L^p}}{\| f \|_{L^2}} \sim \delta^{1/2} \qquad \mbox{for $p>\frac{2d}{d-1}$}.
$$
\end{itemize}

\bibliographystyle{abbrv}
\bibliography{references}

\begin{thebibliography}{10}

\bibitem{BS}
M.~D. Blair and C.~D. Sogge.
\newblock Logarithmic improvements in {$L^p$} bounds for eigenfunctions at the
  critical exponent in the presence of nonpositive curvature.
\newblock {\em Invent. Math.}, 217(2):703--748, 2019.

\bibitem{Bourgain}
J.~Bourgain.
\newblock Eigenfunction bounds for the {L}aplacian on the {$n$}-torus.
\newblock {\em Internat. Math. Res. Notices}, (3):61--66, 1993.

\bibitem{Bourgain2}
J.~Bourgain.
\newblock Moment inequalities for trigonometric polynomials with spectrum in
  curved hypersurfaces.
\newblock {\em Israel J. Math.}, 193(1):441--458, 2013.

\bibitem{BBZ}
J.~Bourgain, N.~Burq, and M.~Zworski.
\newblock Control for {S}chr\"{o}dinger operators on 2-tori: rough potentials.
\newblock {\em J. Eur. Math. Soc. (JEMS)}, 15(5):1597--1628, 2013.

\bibitem{BourgainDemeter1}
J.~Bourgain and C.~Demeter.
\newblock Improved estimates for the discrete {F}ourier restriction to the
  higher dimensional sphere.
\newblock {\em Illinois J. Math.}, 57(1):213--227, 2013.

\bibitem{BourgainDemeter2}
J.~Bourgain and C.~Demeter.
\newblock New bounds for the discrete {F}ourier restriction to the sphere in
  4{D} and 5{D}.
\newblock {\em Int. Math. Res. Not. IMRN}, (11):3150--3184, 2015.

\bibitem{BourgainDemeter3}
J.~Bourgain and C.~Demeter.
\newblock The proof of the {$l^2$} decoupling conjecture.
\newblock {\em Ann. of Math. (2)}, 182(1):351--389, 2015.

\bibitem{BSSY}
J.~Bourgain, P.~Shao, C.~D. Sogge, and X.~Yao.
\newblock On {$L^p$}-resolvent estimates and the density of eigenvalues for
  compact {R}iemannian manifolds.
\newblock {\em Comm. Math. Phys.}, 333(3):1483--1527, 2015.

\bibitem{casselsIntroduction}
J.~W.~S. Cassels.
\newblock {\em An introduction to the geometry of numbers}.
\newblock Die Grundlehren der mathematischen Wissenschaften in
  Einzeldarstellungen mit besonderer Ber\"ucksichtigung der Anwendungsgebiete,
  Bd. 99 Springer-Verlag, Berlin-G\"ottingen-Heidelberg, 1959.

\bibitem{CCU}
F.~Chamizo, E.~Crist\'{o}bal, and A.~Ubis.
\newblock Lattice points in rational ellipsoids.
\newblock {\em J. Math. Anal. Appl.}, 350(1):283--289, 2009.

\bibitem{Cuenin}
J.-C. Cuenin.
\newblock From spectral cluster to uniform resolvent estimates on compact
  manifolds.
\newblock {\em arXiv preprint}, (2011.07254).

\bibitem{DKS}
D.~Dos Santos~Ferreira, C.~E. Kenig, and M.~Salo.
\newblock On {$L^p$} resolvent estimates for {L}aplace-{B}eltrami operators on
  compact manifolds.
\newblock {\em Forum Math.}, 26(3):815--849, 2014.

\bibitem{GL}
P.~Germain and T.~Leger.
\newblock Spectral projectors, resolvent, and fourier restriction on the
  hyperbolic space.
\newblock {\em arXiv preprint}, (2104.04126).

\bibitem{Goetze}
F.~G\"{o}tze.
\newblock Lattice point problems and values of quadratic forms.
\newblock {\em Invent. Math.}, 157(1):195--226, 2004.

\bibitem{Guo2012}
J.~Guo.
\newblock On lattice points in large convex bodies.
\newblock {\em Acta Arith.}, 151(1):83--108, 2012.

\bibitem{Hickman}
J.~Hickman.
\newblock Uniform $l^p$ resolvent estimates on the torus.
\newblock {\em arXiv preprint}, (1907.08131).

\bibitem{Huxley03}
M.~N. Huxley.
\newblock Exponential sums and lattice points. {III}.
\newblock {\em Proc. London Math. Soc. (3)}, 87(3):591--609, 2003.

\bibitem{KRS}
C.~E. Kenig, A.~Ruiz, and C.~D. Sogge.
\newblock Uniform {S}obolev inequalities and unique continuation for second
  order constant coefficient differential operators.
\newblock {\em Duke Math. J.}, 55(2):329--347, 1987.

\bibitem{Kraetzel00}
E.~Kr\"{a}tzel.
\newblock {\em Analytische {F}unktionen in der {Z}ahlentheorie}, volume 139 of
  {\em Teubner-Texte zur Mathematik [Teubner Texts in Mathematics]}.
\newblock B. G. Teubner, Stuttgart, 2000.

\bibitem{Landau15}
E.~Landau.
\newblock Zur analytischen zahlentheorie der definiten quadratischen formen
  (ueber die gitterpunkte in einem mehrdimensionalen ellipsoid.

\bibitem{Nowak2014}
W.~G. Nowak.
\newblock {\em Integer Points in Large Bodies}, pages 583--599.
\newblock Springer International Publishing, Cham, 2014.

\bibitem{Shen}
Z.~Shen.
\newblock On absolute continuity of the periodic {S}chr\"{o}dinger operators.
\newblock {\em Internat. Math. Res. Notices}, (1):1--31, 2001.

\bibitem{Sogge}
C.~D. Sogge.
\newblock {\em Fourier integrals in classical analysis}, volume 105 of {\em
  Cambridge Tracts in Mathematics}.
\newblock Cambridge University Press, Cambridge, 1993.

\bibitem{STZ}
C.~D. Sogge, J.~A. Toth, and S.~Zelditch.
\newblock About the blowup of quasimodes on {R}iemannian manifolds.
\newblock {\em J. Geom. Anal.}, 21(1):150--173, 2011.

\bibitem{Stein}
E.~M. Stein.
\newblock {\em Harmonic analysis: real-variable methods, orthogonality, and
  oscillatory integrals}, volume~43 of {\em Princeton Mathematical Series}.
\newblock Princeton University Press, Princeton, NJ, 1993.
\newblock With the assistance of Timothy S. Murphy, Monographs in Harmonic
  Analysis, III.

\bibitem{Tao}
T.~Tao.
\newblock 247b, notes 2: decoupling theory.
\newblock {\em
  https://terrytao.wordpress.com/2020/04/13/247b-notes-2-decoupling-theory/}.

\bibitem{SteinTomas}
P.~A. Tomas.
\newblock A restriction theorem for the {F}ourier transform.
\newblock {\em Bull. Amer. Math. Soc.}, 81:477--478, 1975.

\bibitem{walfisz1960}
A.~Walfisz.
\newblock \"{U}ber {G}itterpunkte in vierdimensionalen {E}llipsoiden.
\newblock {\em Math. Z.}, 72:259--278, 1959/60.

\end{thebibliography}

\end{document}